\scrollmode
\pdfoutput=1
\documentclass[12pt]{amsart}
\usepackage{graphicx}
\usepackage{amssymb}
\usepackage{psfrag}
\usepackage{mathrsfs}
\usepackage{subfigure}
\usepackage{cite}
\usepackage{xspace}

\newif\iffigs
\figstrue
\figsfalse

\newcommand{\fH}{\mathfrak{H}}

\newcommand{\sC}{\mathscr{C}}

\newcommand{\sR}{\mathscr{R}}
\newcommand{\sL}{\mathscr{L}}

\newcommand{\cA}{{\mathscr{A}}}
\newcommand{\cS}{{\mathscr{S}}}
\newcommand{\sT}{{\mathscr{T}}}
\newcommand{\T}{{\mathscr{T}}}

\newcommand{\He}{{\mathscr{H}}}

\newcommand{\D}{{\mathscr F}}

\newcommand{\R}{\mathbb{R}}
\newcommand{\Z}{\mathbb{Z}}

\newcommand{\NN}{\mathbb{N}}

\newcommand{\ttL}{\mathtt{L}}

\newcommand{\Reduced}{\mathtt{Red}}
\newcommand{\join}[2]{{#1\!\!\relbar\protect\joinrel\protect\relbar\!\!#2}}
\renewcommand{\tilde}[1]{\widetilde{#1}}

\renewcommand{\emptyset}{\varnothing}

\renewcommand{\bar}[1]{\overline{#1}}
\newcommand{\thmcite}[1]{\emph{\cite{#1}}}
\newcommand{\rfthmcite}[2]{\emph{\cite[#1]{#2}}}

\DeclareMathOperator{\Cay}{Cay}
\swapnumbers
\theoremstyle{plain}
\newtheorem{theorem}[subsection]{Theorem}%[section]
\newtheorem{conjecture}[subsection]{Conjecture}%[section]
\newtheorem{proposition}[subsection]{Proposition}%[section]
\newtheorem{corollary}[subsection]{Corollary}%[section]
 
\theoremstyle{definition}
\newtheorem{definition}[subsection]{Definition}%[section]
\newtheorem{example}[subsection]{Example}%[section]
\theoremstyle{remark}
\newtheorem{remark}[subsection]{Remark}%[section]
\date{\today}

\title[Cells and automata]{Kazhdan--Lusztig Cells in planar hyperbolic Coxeter groups and automata}

\author[M. V. Belolipetsky]{Mikhail V. Belolipetsky}
\address{
IMPA\\
Estrada Dona Castorina 110\\
22460-320 Rio de Janeiro, Brazil}
\email{mbel@impa.br}

\author[P. E. Gunnells]{Paul E. Gunnells} 
\address{University of
Massachusetts Amherst\\ Amherst, MA 01003}
\email{gunnells@math.umass.edu}

\author[R. Scott]{Richard Scott}
\address{Department of Mathematics and Computer Science\\
Santa Clara University\\
Santa Clara, CA  95053}
\email{rscott@scu.edu}

\subjclass[2000]{Primary 20F10, 20F55}
\keywords{Kazhdan--Lusztig cells, hyperbolic groups, finite state automata}
\thanks{MB was partially supported by CNPq. PG was partially supported
by NSF.  We thank an anonymous referee for helpful comments}

\begin{document}
\begin{abstract}
Let $C$ be a one- or two-sided Kazhdan--Lusztig cell in a Coxeter
group $(W,S)$, and let $\Reduced (C)$ be the set of reduced
expressions of all $w\in C$, regarded as a language over the alphabet
$S$.  Casselman has conjectured that $\Reduced (C)$ is regular.  In
this paper we give a conjectural description of the cells when $W$ is
the group corresponding to a hyperbolic polygon, and show that our
conjectures imply Casselman's.
\end{abstract}

\maketitle 

%%%%%%%%%%%%%%%%%%
%                %
%  introduction  %
%                %
%%%%%%%%%%%%%%%%%%
\section{Introduction}\label{s:intro}

Let $W$ be a Coxeter group with generating set $S$.  In their study of
representations of Coxeter groups and Hecke algebras, Kazhdan and
Lusztig introduced the decomposition of $W$ into \emph{cells}
\cite{kl}.  The cells are equivalence classes in $W$ determined by the
left and right descent sets of elements of $W$ and the degrees of the
Kazhdan--Lusztig polynomials $P_{x,y}$ (\S\ref{s:defs}).  Today cells
are known to have many applications in representation theory; for some
references, see the bibliography of \cite{notices}.

This paper addresses the computability of the cells, in the following
sense.  Given a cell $C$, one can ask for an efficient way to encode
its elements.  Since elements of $W$ are easily represented by reduced
expressions in the generators $S$, it is natural to ask for a solution
in terms of such expressions.  However, since the definition of the
cells involves a complicated equivalence relation, it is certainly not
clear that this is possible.

Despite this, W.~Casselman has conjectured that cells can be
efficiently encoded.  To state his conjecture, we need some
terminology from the theory of formal languages; for more information
see \cite{aho}.

Let $A$ be a finite alphabet of characters.  By a \emph{language}
$\ttL$ over $A$ we mean a collection of finite-length ordered words
built from elements of $A$.  A \emph{finite state automaton} $\cA$
with alphabet $A$ is a finite directed graph on a vertex set $\cS $,
called states, with edges labeled by elements of $A \cup \{\varepsilon
\}$.  Different edges leaving a given vertex are assumed to have
different labels.  One vertex is defined to be the \emph{initial
state}; a subset of $\cS$ is chosen and defined to be the
\emph{accepting states}.  A finite state automaton encodes certain
words built from $A$ through path traversal: one starts at the initial
state and follows a directed path of any length that terminates at an
accepting state.  As the path is traversed the vertex labels are
concatenated into a word (the symbol $\varepsilon$ represents a
``null-transition;'' the word is unaltered if $\varepsilon$ is read).
The collection of words that can be so constructed forms the language
recognized by $\cA$.  A language is called \emph{regular} if it can be
recognized by a finite state automaton.

Regular languages are the simplest infinite languages one encounters
in the hierarchy of formal languages.  Many languages in algebra are
regular.  For instance, via an earlier paper of Davis--Shapiro
\cite{ds}, work of Brink--Howlett implies that the language $\Reduced
(W)$ of all reduced expressions in the generators $S$ is regular
\cite{bh}.  Any cell $C$ induces a sub language $\Reduced (C)\subset
\Reduced (W)$, namely all the reduced expressions of elements in $C$.
We can now state Casselman's conjecture:
\begin{conjecture}\label{conj:casselman}
For any Coxeter group $W$ and any (two- or one-sided) cell $C\subset
W$, the language $\Reduced (C)$ is regular.
\end{conjecture}

Casselman's conjecture is known to be true for affine Weyl groups from
earlier work of one of us (PG) \cite{automata}.  In this paper we
investigate the case that $( W, S)$ is a Coxeter group corresponding
to a hyperbolic polygon.  In other words, $W$ can be realized as the
discrete subgroup of isometries of the hyperbolic plane $\fH$
generated by the reflections through the side of a geodesic polygon.
The cells of such groups have been considered earlier by B\'edard
\cite{bedard1, bedard2} and one of us (MB) \cite{misha}. We state
conjectures due to two of us (MB and PG) that describes the
Kazhdan--Lusztig cells of $W$ in terms of reduced expressions.  Then
we prove (assuming the conjectures) that for any left, right, or
2-sided Kazhdan--Lusztig cell $C$, the language $\Reduced (C)$ is
regular.  Moreover, when combined with previous work of two of us (MB
and PG), the results in this paper prove the regularity of cells for
certain Coxeter groups (cf.~Remark \ref{rem:otherstuff}). We note that
the proofs in this paper use word-hyperbolicity of $W$ in an essential
way, and in particular do not apply to affine Weyl groups.

We now give an overview of the paper.  In \S \ref{s:defs} we give
background on Coxeter groups and recall the definition of
Kazhdan--Lusztig cells. Section \ref{s:conjs} states conjectures for
cells in Coxeter groups attached to tessellations of the hyperbolic
plane by polygons.  In \S\ref{s:wordhyp} we give background on
word hyperbolic groups and state the results we need from geometric
group theory.  Finally \S \ref{s:mainresults} gives our main results.

\section{Definitions and basic examples}\label{s:defs}

In this section we recall the basics of Coxeter groups and define
Kazhdan--Lusztig cells.  For more details we refer to
\cite{humph.book, bb, kl}.  

A \emph{Coxeter group} $W$ is a group generated by a finite subset
$S\subset W$ where the defining relations have the form $(st)^{m
(s,t)}=1$ for pairs of generators $s,t\in S$.  The exponents $m(s,t)$
are taken from $\NN\cup \{\infty \}$, and we require $m(s,s) = 1$, so
that each generator $s$ is an involution.  Let $I\subset S$ be a
subset of the generators.  The subgroup of $W$ generated by $I$ is
called a parabolic subgroup and is denoted $W_{I}$.

Any representation of $w\in W$ as a product of generators is called an
\emph{expression}.  An expression is called \emph{reduced} if it
cannot be made shorter by applying the defining relations of $W$.  The
length of a shortest expression for $w$ is denoted $l (w)$. For any
$w\in W$, we define the \emph{left descent set} $\sL (w)\subset S$ to
consist of those $s\in S$ such that $l (sw)< l (w)$.  We similarly
define the \emph{right descent set} $\sR(w)$ to be those $s$ such that
$l (ws)<l (w)$.  For $w,u,v\in W$, we write $w=u.v$ if $w=uv$ and $l
(w) = l (u)+l (v)$.

Given an expression $s_{1}\dotsb s_{N}$, a \emph{subexpression} is a
(possibly empty) expression of the form $s_{i_{1}}\dotsb s_{i_{M}}$,
where $1\leq i_{1}<\dotsb <i_{M}\leq N$.  The \emph{Chevalley--Bruhat
order} is the partial order on $W$ defined by putting $v\leq w$ if an
expression for $v$ appears as a subexpression of a reduced expression
for $w$.  Given any $v,w\in W$, let $[v,w]$ be the interval between
$v$ and $w$, that is $[v,w]=\{ x\in W\mid v\leq x\leq w\}$.

The Kazhdan--Lusztig polynomials are most easily defined in terms of
an auxiliary family of polynomials, the \emph{$R$-polynomials}.  
This family $\{R_{v,w} (q)\in \Z [q]\mid v,w \in
W \}$ is defined to be the unique collection of polynomials satisfying
the following properties (cf.\cite[Theorem 5.1.1]{bb}): (i)
$R_{v,w}(q)=0$ if $v\not \leq w$; (ii) $R_{v,w} (q) =1$ if $v=w$; and
(iii) if $s\in\sR (w)$, then $R_{v,w} (q)= R_{vs,ws} (q)$ if $s\in \sR
(v)$, and is $qR_{vs,ws} (q) + (q-1) R_{v,ws} (q)$ otherwise.  Given
the $R$-polynomials, 
the \emph{Kazhdan--Lusztig polynomials} $P_{v,w} (q)$ can be described as the
unique family of polynomials satisfying (cf.\cite[Theorem 5.1.4]{bb})
(i) $P_{v,w} (q)=0$ if $v\not \leq w$; (ii) $P_{v,w} (q)=1$ if $v=w$;
(iii) $\deg P_{v,w} (q)\leq (l (w)-l (v)-1)/2$ if $v<w$; and
(iv) $q^{l (w)-l (v)}P_{v,w} (q^{-1}) = \sum_{x\in [v,w]}
R_{v,x} (q)P_{x,w} (q)$ if $v\leq w$.  If $v<w$, we write $\mu (v,w)$
for the coefficient of $q^{(l (w)-l (v)-1)/2}$ in $P_{v,w} (q)$.
We write $\join{v}{w}$ and $\join{w}{v}$ if $\mu (v,w)\not =0$.

We are finally ready to define cells.  The \emph{left $W$-graph}
$\Gamma_{\sL}$ of $W$ is the directed graph with vertex set $W$, and
with an arrow from $v$ to $w$ if and only if $\join{v}{w}$ and $\sL
(v)\not\subset \sL (w)$.  The left cells are extracted from
$\Gamma_{\sL}$ as follows.  Given any directed graph, we say two vertices
are in the same \emph{strong connected component} if there exist
directed paths from each vertex to the other.  Then the \emph{left
cells} of $W$ are exactly the strong connected components of the graph
$\Gamma_{\sL}$.  The \emph{right cells} are defined using the
analogously constructed \emph{right $W$-graph} $\Gamma_{\sR}$.  We say
$v,w$ are in the same \emph{two-sided cell} if we can find a sequence
$v=w_{1},w_{2},\dotsc ,w_{k}=w$ such that $w_{i}, w_{i+1}$ lie in
either the same left or right cell.

We need one final ingredient to state our conjecture in the next
section: the $a$-function.

Let $\He$ denote the Hecke algebra of $W$ over the ring $\cA =
\Z[q^{1/2}, q^{-1/2}]$ of Laurent polynomials in $q^{1/2}$. This
algebra is a free $\cA$-module with a basis $\sT = \{ T_w \mid w\in
W\}$ and with multiplication determined by $ T_wT_{w'} = T_{ww'}$ if
$l (ww') = l (w)+l (w')$, and $T_s^2 = q + (q-1)T_s $ for
$s\in S$.  Together with the basis $\sT$, we can define in $\He$
another basis $\sC = \{ C_w \mid {w\in W}\}$.  The element $C_w\in
\sC$ can be expressed in terms of $\sT$ and the Kazhdan--Lusztig
polynomials by
$$C_w = \sum_{y\le w}(-1)^{l (w)-l (y)}q^{l (w)/2-l (y)}
P_{y,w}(q^{-1})T_y.$$ 

Now consider the multiplication of the $\sC$-basis elements in
$\He$. We can write
$$C_x C_y = \sum_z h_{x,y,z} C_z,\ h_{x,y,z}\in\cA.$$ Let $a(z)$ be
the smallest integer such that $q^{a(z)/2}h_{x,y,z} \in \cA^+$ for all
$x, y \in W$, where $\cA^+ = \Z[q^{1/2}]$.  It is a standard
conjecture that $\{a (w)\mid w\in W \}\subset \Z$ is bounded for any
Coxeter group.  The $a$-function was introduced by Lusztig in
\cite{lusztig}, where he proved this conjecture for affine Weyl
groups. In \cite{misha} it was shown that the $a$-function is bounded for
right-angled Coxeter groups.  N.~Xi recently showed that the
$a$-function is bounded for Coxeter groups with complete Coxeter
graphs (i.e.~no two generators commute) \cite{xi}; this paper has
further ramifications for our current article, see Theorem
\ref{thm:xi}.  P.~Zhou has recently proved that the $a$-function is
bounded if $W$ has rank $3$ \cite{zhou}.

\section{Conjectures about cells of hyperbolic polygon groups}\label{s:conjs}

In this paper we take $W$ to be a hyperbolic polygon group.  This
means the following.  Let $\fH$ be the hyperbolic plane, and let
$\Delta \subset \fH $ be an $n$-sided geodesic polygon with angles
$\alpha_{i} = \pi/a_{i}$, $i=1,\dotsc ,n$.  (We omit the conditions
the denominators $a_{i}$ satisfy to make $\Delta$ hyperbolic; we also
allow the angles to vanish, in which case the polygon has ideal
vertices.)  Label the sides of $\Delta$ by $\sigma_{1}, \dotsc ,
\sigma_{n}$, such that the angle $\alpha_{i}$ sits between the sides
$\sigma_{i}, \sigma_{i+1}$, and where the subscripts are taken mod $n$
as necessary.  Then the generating set $S$ of $W$ has $n$ elements
$s_{1},\dotsc s_{n}$, corresponding to the sides $\sigma_{i}$.  We put
$m (s_{i},s_{j}) = \infty$ unless $\sigma _{i}$ and $\sigma_{j}$ meet
at the angle $\alpha_{k} \not = 0$.  In the latter case we put $m
(s_{i},s_{j}) = a_{k}$.  

It is not hard to see that $W$ is isomorphic to the discrete subgroup
of isometries of $\fH$ generated by reflections in the lines through
the $\sigma_{i}$.  Thus there is an action of $W$ on $\fH$ by
reflections, the polygon $\Delta$ is a fundamental domain, and the
translates $\{ w\cdot \Delta \mid w\in W\}$ form a
tessellation of $\fH$ (note our convention that the reflection action
of $W$ on $\fH$ is a left action).  The correspondence $w\mapsto
w\cdot \Delta$ is a bijection between $W$ and the tiles in the
tessellation.  Using this we identify $W$ with the set of all tiles.

We can also use this identification to define certain subsets of $W$.
Recall that $\sL (w)$ denotes the set of left descents of an element
$w$.  Given any subset $T\subset S$, we let $W^{T}$ be the (possibly
empty) set of all $w\in W$ such that $\sL (w) = T$.  The tessellation
allows us to identify the sets $W^{T}$ as follows.  First,
$W^{\emptyset}$ consists of $\Delta$ itself.  Next, any edge of
$\Delta$ corresponds to a generator $s\in S$.  Extending this edge to
a line divides the plane $\fH$ into two half-spaces, one containing
$\Delta$ and one not.  The half-space $H_{s}$ not containing $\Delta$
contains all elements $w$ such that $s \in\sL (w)$.  Any (non-ideal)
vertex of $\Delta$ corresponds to an order 2 subset $T$ with
$W^{T}\not =\emptyset$.  Namely we have $W^{T} = H_{s}\cap H_{s'}$,
where $T=\{s,s'\}$ and $s,s'$ label the edges of $\Delta$ meeting this
vertex.  Finally, if $T,T'$ have order $2$ and $T\cap T' = \{s \}$ has
order $1$, then $W^{\{s \}}$ consists of $H_{s}\smallsetminus
W^{T}\cap W^{T'}$.  These give all subsets $T$ such that $W^{T}\not
=\emptyset$.

We call a subgroup $D\subset W$ \emph{finite dihedral} if $D$ is the
parabolic subgroup for an order $2$ subset $T$ with $W^{T}$ nonempty.
Let $\D$ be the set of finite dihedral subgroups and let $\T$ be the
set of order $2$ subsets indexing them.  Assume that the distinct
nonzero exponents are $e_{1}<e_{2}<\dotsb <e_{m}$, where $m\leq n$.
This means there are $m$ isomorphism classes of dihedral subgroups of
$W$.  We write $\D = \D_{1}\cup \dotsb \cup \D_m$, where $\D_{i}$ is
the set of finite dihedral subgroups of exponent $e_{i}$.  We also let
$\T = \T_{1}\cup \dotsb \cup \T_{m}$ be the corresponding partition of
$\T$. For any $D\in \D$, let $w_{D} \in D$ be the longest element.
For $i=1,\dotsc ,m$ let $W_{i} = \{w_{D}\mid D\in \D_{i} \} $ be the
sets of longest elements.  We also write $w_{T}$ for $w_{D}$ if $D =
\langle T\rangle$.

We are now ready to give a conjectural description of the two-sided
cells of $W$.  We define a sequence of subsets $C_{m},\dotsc ,C_{1}$
of $W$ as follows.  First, $C_{m}$ is defined by 
\[
w\in C_{m} \text{\ if and only if\ } w=u.w_{D}.v \text{\ for some\ } w_{D}\in W_{m}.
\]
In other words, $w$ is in $C_{m}$ if and only if there is some reduced
expression of $w$ that contains a reduced expression for $w_{D}$ as a
subword, where $D$ is finite dihedral subgroup of maximal
exponent $e_{m}$.  Next, for $i<m$ we define $C_{i}$ by
\begin{multline*}\label{eq:defofCi}
w\in C_{i} \text{\ if and only if\ }  w=u.w_{D}.v \text{\ for some\ } w_{D}\in W_{i},\\
\text{and\ } w\not =x.w_{D'}.y \text{\ for any\ } w_{D'}\in W_{k} \text{\ with\ } k>i.
\end{multline*}
Thus $w\in C_{i}$ if it has a reduced expression containing a
subword of $w_{D}$ with $D$ finite dihedral of exponent $e_{i}$,
and has \emph{no} reduced expression containing any $w_{D'}$ as a
subword with $D'$ finite dihedral of exponent $>e_{i}$.  We also
define subsets $C_{\text{id}} = \{\text{id} \}$ and
\[
C_{0} =
\{w\mid \text{$w$ has a unique reduced expression} \}.
\]
Let $\sC$ be the collection $\{C_{\text{id}},C_{0},\dotsc  ,
C_{m}\}$.

\begin{conjecture}\label{conj:2sidedcells}
\begin{enumerate}
\item The decomposition $\sC$ gives the partition of $W$ into two-sided cells.
\item On the two-sided cell $C_{i}$, the $a$-function equals the
length of any element of $W_{i}$.
\end{enumerate}
\end{conjecture}

\begin{example}\label{ex:237-2sided}
To illustrate Conjecture \ref{conj:2sidedcells}, we consider the
hyperbolic triangle group
\[
W = W_{237} = \langle r, s, t \mid  r^{2} = s^{2}=t^{2} =
(rs)^{3} = (rt)^{2} = (st)^{7} = 1\rangle.
\]
There are three finite dihedral subgroups, of orders $4, 6, 14$,
corresponding to the exponents $2,3,7$. The longest words in these
subgroups have reduced expressions $rt$, $rsr$, and $stststs$.  One
knows from the theory of Coxeter groups that one can pass between any
two reduced expressions of a given element by applying the
substitutions $rt=tr$, $rsr=srs$, and $stststs=tststst$. 

Figure \ref{fig:237} shows the partition $\sC$ for $W$.  The simplest
subsets are the white and yellow triangles.  The white
triangle corresponds to the identity in $W$ and is the set
$C_{\text{id}}$.  The yellow triangles correspond to those elements of
$W$ with unique reduced expressions, such as $r$, $s$, $t$, $rs$, and
$srt$.  There are $27$ such elements in $W$, and together they form the
set $C_{0}$. 

The blue triangles are the set $C_{1}$.  They consist of elements of
$W$ with a reduced expression containing $rt$ as a subword, but
with no reduced expressions containing either $rsr$ or $stststs$ as
subwords.  One way to think about what this means is the
following.  Elements of $C_{1}$ do not have unique reduced
expressions, but they almost do: when moving between reduced
expressions for these elements, one only applies relations of the form
$rt =tr$, and never uses the other two relations $rsr=srs$ and
$stststs=tststst$.  Thus the blue triangles correspond to those
elements of $W$ with the simplest possible nonunique reduced
expressions.  Elements in a Coxeter group such that any two reduced
expressions are connected by those moves that exchange two adjacent
commuting generators are called \emph{fully commutative} in the
literature \cite{stembridge}; their special relationship with
Kazhdan--Lusztig cells was investigated by Green--Losonczy
\cite{green.fc} and Shi \cite{shi.fc}.

Next we have the green triangles, which form the set $C_{2}$.  These
are the elements with the next most complicated reduced expressions:
when rewriting a reduced expression for any of these elements, one
uses the relations $rt=tr$ and $rsr=srs$, but never the relation
$stststs=tststst$.  Equivalently, no element in the green set has a
reduced expression containing the subword $stststs$, and every
element has at least one reduced expression containing $rsr$ as a
subword.

Finally we come to the set $C_{3}$, which is made of the red
triangles.  Any element in the red set has at least one reduced
expression with $stststs$ as a subword.  

The computation of the two-sided cells of $W$ has unfortunately not
been carried out.  Nevertheless, the partition $\sC$ experimentally
agrees with the two-sided cells in the following sense.  One can
naively compute an approximation to the two-sided cells by computing
many Kazhdan--Lusztig polynomials and descent sets, and thus computing
approximations to the $W$-graphs $\Gamma_{\sL}$ and $\Gamma_{\sR}$;
the only limitations to improving these approximations are computer
time and memory.  Doing this one finds that Figure \ref{fig:237}
agrees with the partition into two-sided cells in a bounded region
around $C_{\text{id}}$.  Actually one can do significantly better than
the naive computation: techniques of \cite{BG1} lead to predictions of
nonzero $\mu$-values for pairs $v,w$ with $|l (w)-l (v)|$ arbitrarily
large.  That these are in fact nonzero can be checked in any given
example by computer, again with the only limitation being computing
resources.  This then leads to much better approximations to the
two-sided cells that continue to agree with the partition $\sC
$.\footnote{We remark that if one could prove that all our claimed
$\mu$-coefficients from \cite{BG1} were nonzero, one could then prove
that $\sC$ matches the decomposition into two-sided cells.}

\begin{figure}[htb]
\begin{center}
\includegraphics[scale=0.65]{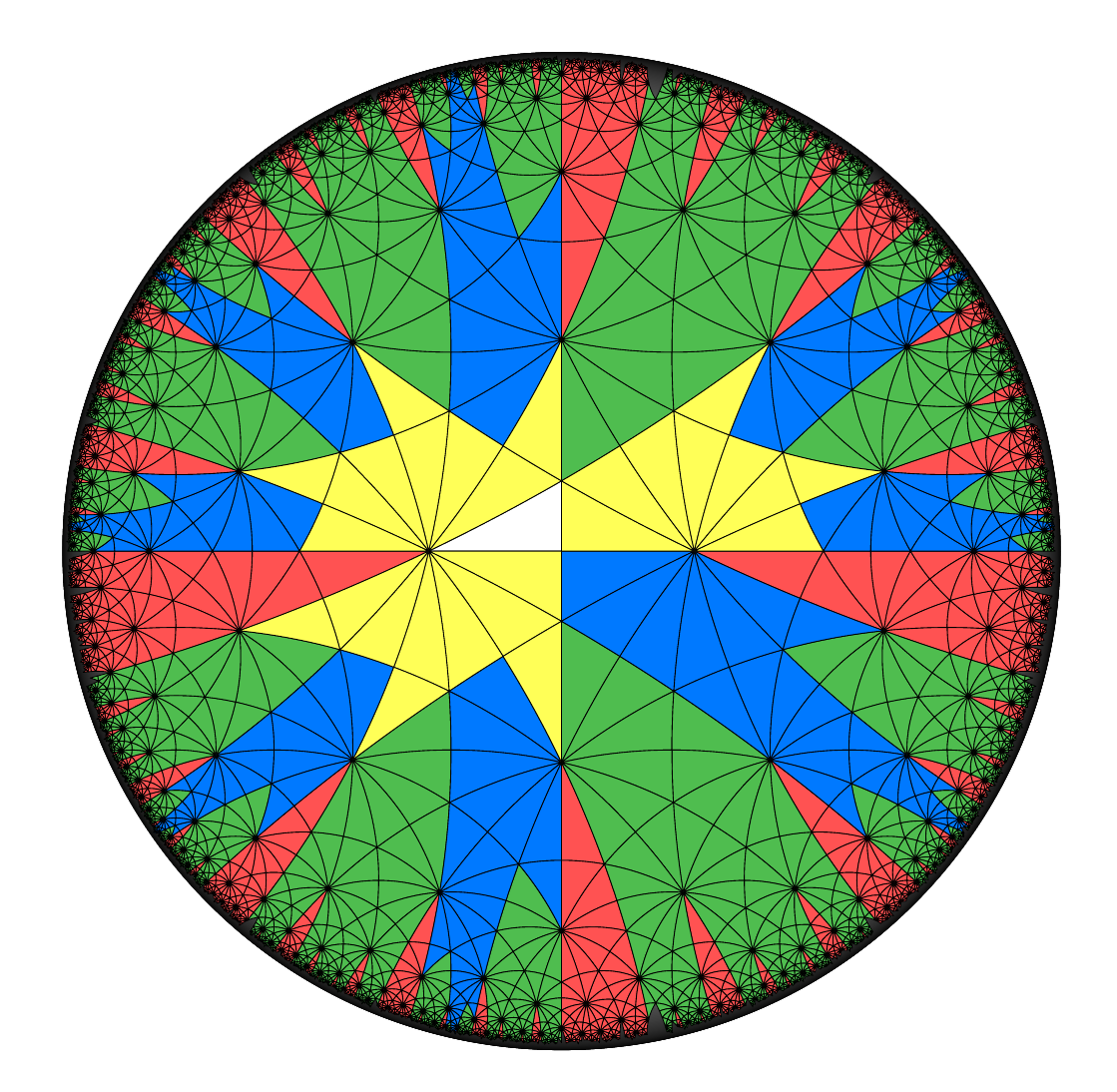}
\end{center}
\caption{\label{fig:237} The alcoves for the triangle group
$W=W_{237}$ and the partition $\sC$.  The sets $C_{\text{id}}, C_{0},
C_{1}, C_{2}, C_{3}$ are (respectively) the white triangle, the yellow
triangles, the blue triangles, the green triangles, and the red
triangles.}
\end{figure}
\end{example}

\begin{example}\label{2224-2sided}
For an examples where some exponents are equal, consider the Coxeter group $W'=W_{2224}$ generated by reflections in
the sides of the hyperbolic quadrilateral with angles $\pi /2, \pi /2,
\pi /2, \pi /4$.  This group has the presentation 
\[
W'=W_{2224} = \langle a, b, c, d \mid  a^{2} = b^{2}=c^{2} = d^{2} =
(ab)^{2} = (bc)^{2} = (cd)^{2} = (ad)^{4} = 1\rangle.
\]
There are four finite dihedral subgroups of orders $4$, $4$, $4$, and
$8$.  Since there are two distinct exponents, Conjecture
\ref{conj:2sidedcells} predicts that there are four two-sided cells.
These are the four sets 
\[
C_{\text{id}}, C_{0}, C_{1}, C_{2}
\]
shown in Figure \ref{fig:2224-2sided}; the colors are (respectively)
white, yellow, blue, and green.

\begin{figure}[htb]
\begin{center}
\includegraphics[scale=0.65]{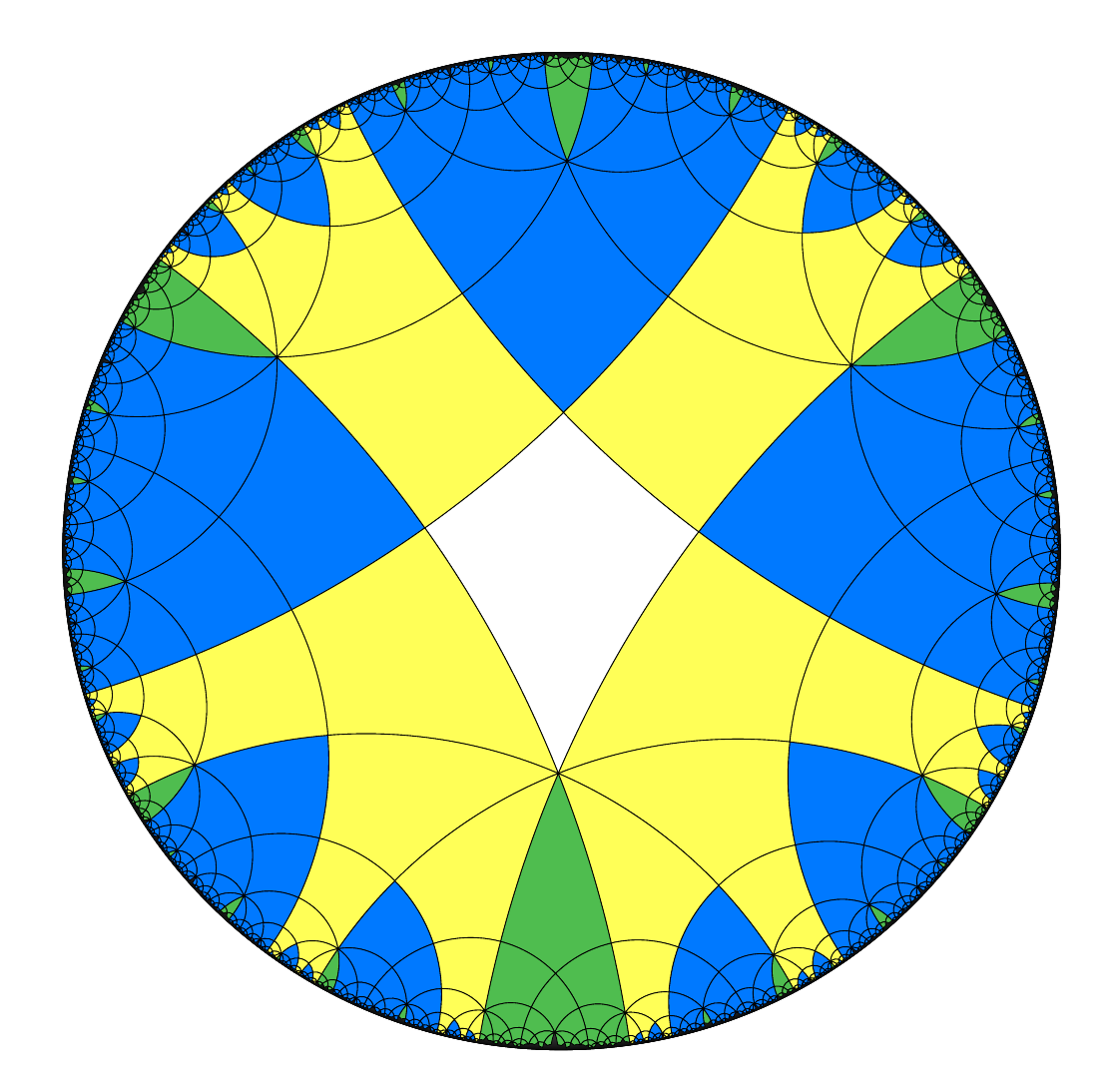}
\end{center}
\caption{The partition $\sC$ determined by Conjecture
\ref{conj:2sidedcells} for the Coxeter group 
$W_{2224}$\label{fig:2224-2sided}.}
\end{figure}
\end{example}

Returning now to the general discussion, we remark that both
$C_{\text{id}}$ and $C_{0}$ are known to be two-sided cells, the
former for trivial reasons and the latter from work of Lusztig
\cite[\S\S 3.7--3.8]{lusztig83}.  It also follows from computations of
Lusztig \cite{lusztig} that the two-sided cells of the planar affine
Weyl groups $\tilde{A}_{2}$, $\tilde{B}_{2}$, $\tilde{G}_{2}$ can be
described by Conjecture \ref{conj:2sidedcells}, even though they are
of course not hyperbolic.  We also have the following theorem of Xi,
which gives confirmation of Conjecture \ref{conj:2sidedcells} for
certain (not necessarily hyperbolic) $W$:

\begin{theorem}\label{thm:xi}\thmcite{xi}
Suppose $W$ is crystallographic and that no exponent of $W$ is $2$.
Then $\sC$ gives the partition of $W$ into two-sided cells. 
\end{theorem}

Next we turn to the one-sided cells.  Given any $T\in \T_{i}$, define 
\begin{equation}\label{eq:defofUT}
U^{T} = W^{T}\smallsetminus \bigcup_{j>i} C_{j} \cap W^{T}.
\end{equation}
In particular if $i=m$, we have $U^{T} = W^{T}$.  Let $\Omega_{T} =
\{w^{-1}w_{T}\mid w\in U^{T}\}$ and let $\Omega_{i} = \bigcup_{T\in
\T_{i}} \Omega_{T}$.  The one-sided cells will be built from the sets
$w\cdot U^{T}$, where $w\in \Omega_{T}$ and $T$ ranges over $\T$.  We
put a partial order on $\Omega_{i}$ by $w\preceq w'$ if $l (w)\leq l
(w')$ and there exists $T,T'\in \T_{i}$ such that $(w\cdot U^{T} ) \subseteq
(w'\cdot U^{T'})$. We define $\Omega_{i}^{\circ}$ to be
the minimal elements in $\Omega_{i}$ with respect to this partial
order, and we write $\Omega_{T}^{\circ }$ for the minimal elements of
$\Omega_{i}$ appearing in $\Omega_{T}$.

\begin{conjecture}\label{conj:1sidedcells}
The subsets $\{w\cdot U^{T} \mid w\in \Omega_{T}^{\circ}, T\in \T_{i}\}$ are the
right cells in $C_{i}$.
\end{conjecture}

\begin{example}\label{ex:237-1sided}
Figures \ref{fig:a}--\ref{fig:b} illustrate Conjecture
\ref{conj:1sidedcells} for $W=W_{237}$.  First we consider the
two-sided cell $C_{3}$, which consists of all the red regions in
Figure \ref{fig:237}.  It is clear from Figure \ref{fig:237} that
$C_{3}$ is a disjoint union of geodesically convex regions in the
hyperbolic plane.  Further, a little experimentation suggests that
each connected component in $C_{3}$ can be taken to any other by a
sequence of reflections.  This is the main motivation behind
Conjecture \ref{conj:1sidedcells}, which uses finitely many one-sided
cells to generate all others.

Let $T\subset S$ be $\{s,t \}$.  The region $U^{T}$ from
\eqref{eq:defofUT} is the connected component of $C_{3}$ in Figure
\ref{fig:237} that shares a vertex with the identity triangle
$C_{\text{id}}$.  This is our initial one-sided cell; we build the
others by taking the reflected images $w^{-1}w_{T}\cdot U^{T}$, where
$w$ ranges over $U^{T}$; this set of prefixes $\{w^{-1}w_{T} \mid w
\in U^{T}\}$ is the set we denote $\Omega_{T}$.  Since there is only
one finite dihedral subgroup of exponent $7$, we have $\Omega_{3} =
\Omega_{T}$.

What happens is in depicted in Figures \ref{fig:a}--\ref{fig:b}.  The
red region in Figure \ref{fig:a}, together with the purple triangles
inside it, is the subset $U^{T}$.  Let $\Delta$ be the purple triangle
that shares a vertex with $C_{\text{id}}$, i.e.~the purple triangle at
the tip of the red region $U^{T}$.  Then all the purple triangles
correspond to $w\cdot \Delta$ as $w$ ranges over $\Omega_{3}$.

Figure \ref{fig:b} shows the regions $\{w\cdot U^{T} \mid w\in
\Omega_{3} \}$ (colored randomly).  What has happened is that the
original red region has been reflected so that its tip has been taken
to one of the purple triangles in Figure \ref{fig:a}.  Thus some
images meet others; indeed, if this happens then one translate of
$U^{T}$ is entirely contained in another.  The partial order that
determines $\Omega_{3}^{\circ}$ selects the reflections corresponding
to the purple triangles in Figure \ref{fig:a} lying at the tips of the
orange regions.  Flipping the original red region by the reflections
in $\Omega^{\circ}_{3}$ recovers all the red regions in Figure
\ref{fig:237}.

\begin{figure}[htb]
\begin{center}
\includegraphics[scale=0.65]{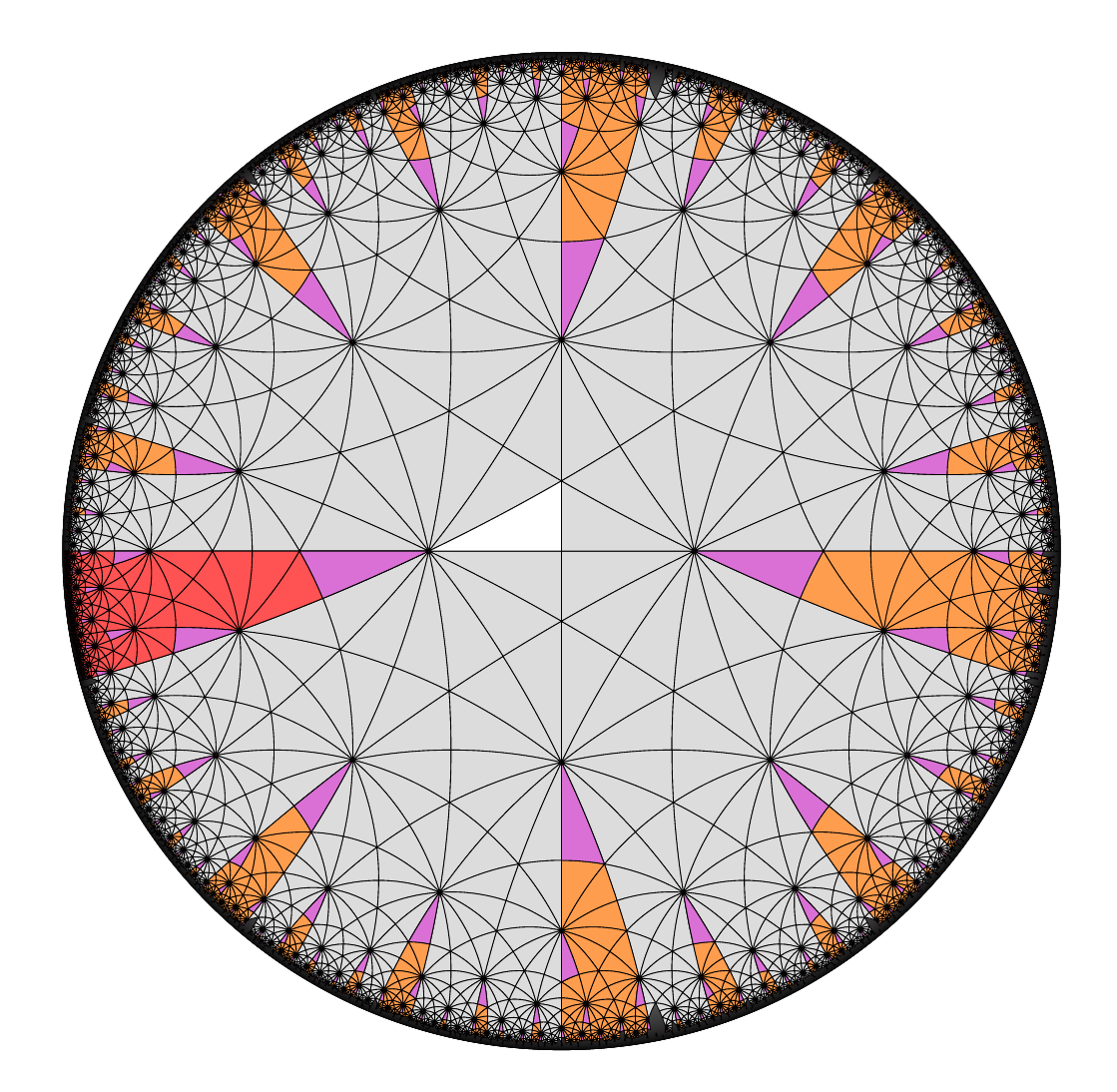}
\end{center}
\caption{The decomposition of the two-sided cell $C_{3}$ in $W_{237}$
into right cells.  We have $T = \{s,t \}$.  The red region, including
the purple trianges inside it, is $U^{T}$.  The purple triangles are
those of the form $w\cdot \Delta$, where $w\in \Omega_{3}$ and
$\Delta$ is the purple triangle at the tip of the red region.
\label{fig:a}}
\end{figure}

\begin{figure}[htb]
\begin{center}
\includegraphics[scale=0.65]{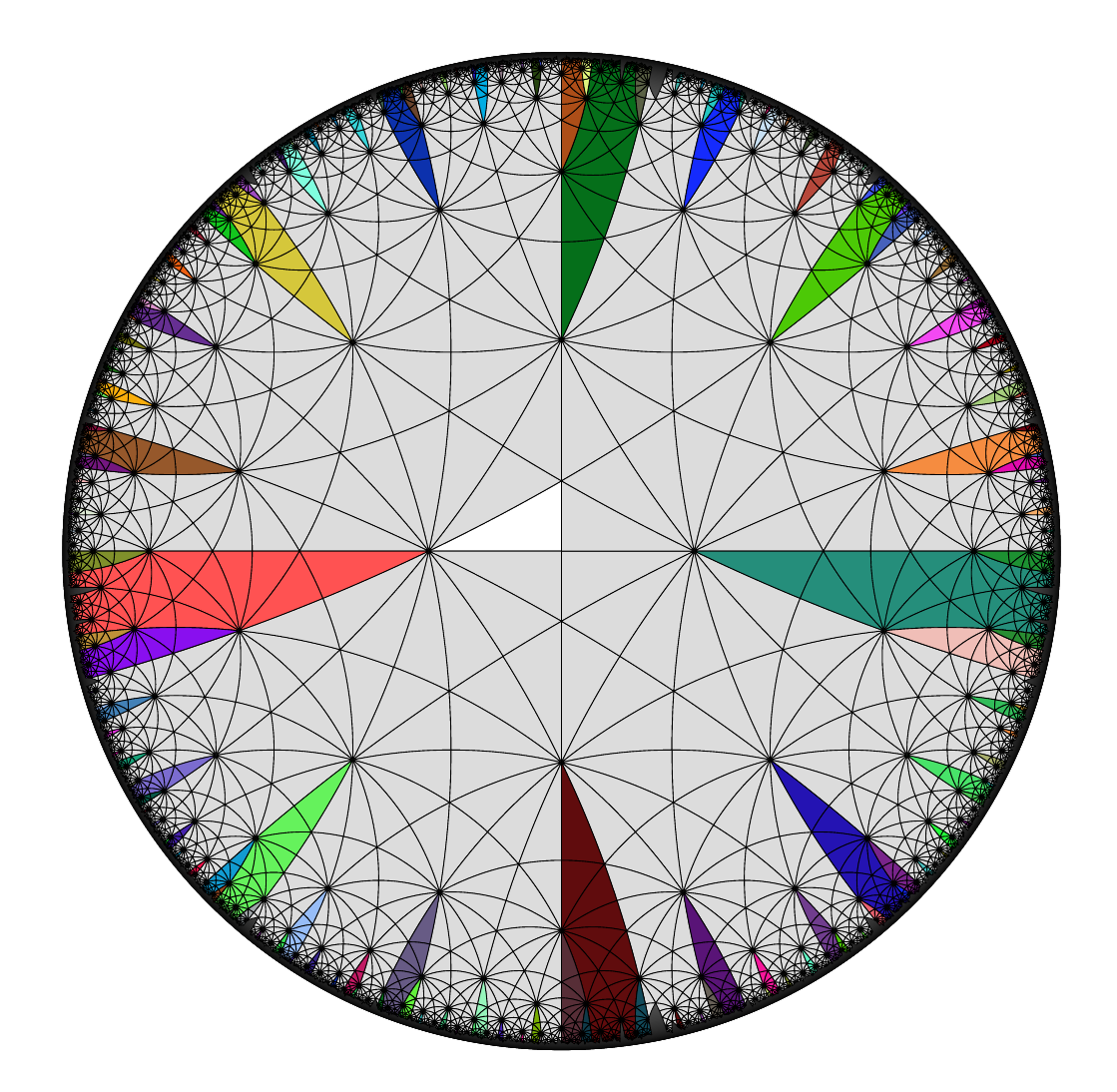}
\end{center}
\caption{The images $w\cdot U^{T}$, as $w$ ranges over $\Omega_{3}$.\label{fig:b}}
\end{figure}

\end{example}

\begin{example}\label{ex:2224-1-sided}

Now we consider the one-sided cells in $W'=W_{2224}$ that form the
blue two-sided cell $C_{1}$ in Figure \ref{fig:2224-2sided}.  This
time there are three relevant dihedral subgroups, since there are
three of exponent $2$: the subsets of the generators are $T = \{
a,b\}$, $T'=\{ b,c\}$, and $T'' = \{ c,d\}$.  We focus on $T$ and
$T''$, since $T'$ can be treated by symmetry.  The region $U^{T}$
(respectively, $U^{T''}$) is the blue region in Figure
\ref{fig:2224-2sided} immediately to the right (respectively, above)
the identity cell $C_{\text{id}}$.  Figures \ref{fig:12}--\ref{fig:23}
are the analogues of Figure \ref{fig:b}.  We see the original regions
(colored the same blue as in Figure \ref{fig:2224-2sided}) and all
their reflected images (colored randomly).  Comparing with Figure
\ref{fig:2224-2sided}, one observes that if one image meets another,
then one is entirely contained in the other, just as in Figure
\ref{fig:b}.

\begin{figure}[htb]
\begin{center}
\includegraphics[scale=0.65]{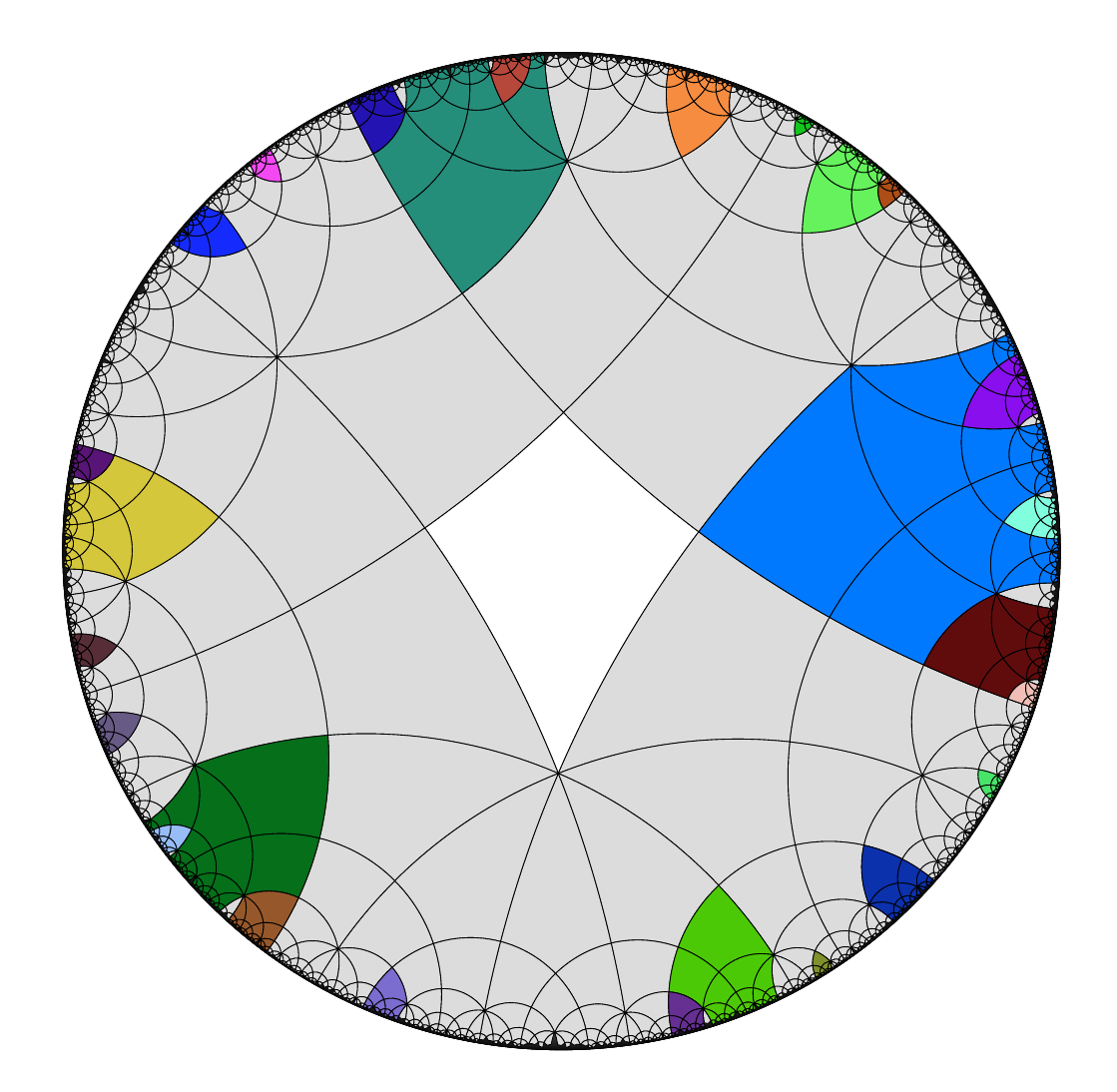}
\end{center}
\caption{The images $w\cdot U^{T}$ where $T=\{a,b \}$ and $w\in
\Omega_{T}$.\label{fig:12}}
\end{figure}

\begin{figure}[htb]
\begin{center}
\includegraphics[scale=0.65]{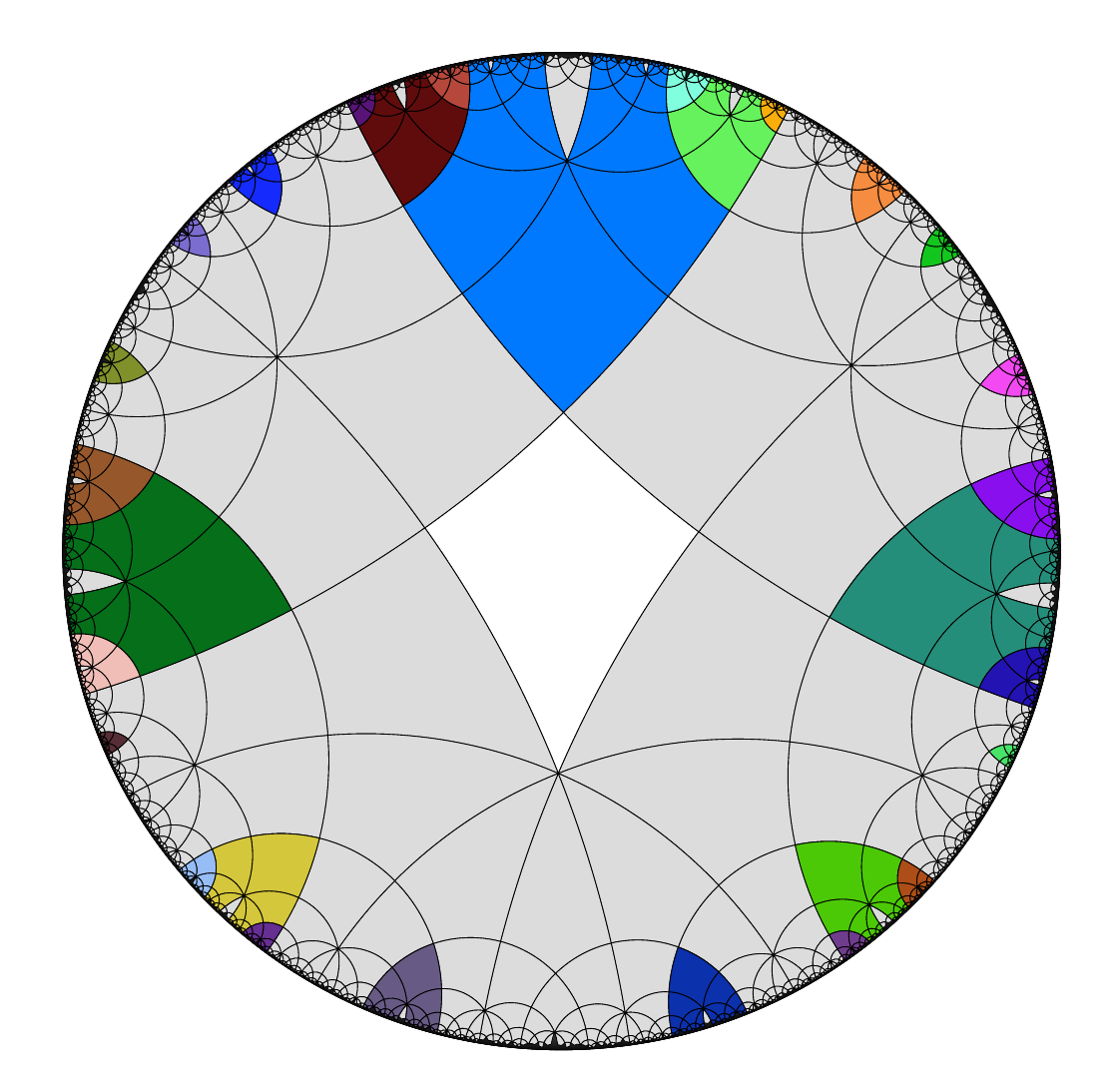}
\end{center}
\caption{The images $w\cdot U^{T''}$ where $T''=\{c,d \}$ and $w\in
\Omega_{T''}$.\label{fig:23}}
\end{figure}

\end{example}

Conjectures \ref{conj:1sidedcells} and \ref{conj:2sidedcells} should
be considered as a special case of conjectures from \cite{BG1} applied to
hyperbolic polygon groups.

% \begin{figure}[htb]
% \centering
% \subfigure[\label{fig:a}$U^{23}$ and $\Omega_{3}$]{\includegraphics[scale=0.5]{pix/cells_cox_237_RED_2s_and_reverse}}
% \quad \quad
% \subfigure[\label{fig:b}The translates $U^{23} \cdot \Omega_{3}$]{\includegraphics[scale=0.5]{pix/flipped_red_cell}}
% \end{figure}

\section{Word hyperbolic groups and automata}\label{s:wordhyp}

In this section we prove that certain languages in word hyperbolic
groups are regular.  We will then apply these results to the languages
$\Reduced(C)$ where $C$ is a Kazhdan--Lusztig cell in a hyperbolic polygon group.
First we define word hyperbolic groups and recall some of the
standard facts we shall need.  Details and additional properties can
be found, for example, in \cite{gromov, bridsonbook}.  

Let $(X,d)$ be a geodesic metric space.  A {\em geodesic triangle}
consists of $3$ points in $X$ together with geodesics joining each
pair of points.   A geodesic triangle is called {\em $\delta$-thin}
($\delta\in\R_{>0}$) if every side is in a $\delta$-neighborhood of
the other two sides.  The metric space $X$ is called {\em
$\delta$-hyperbolic} if  every geodesic triangle is $\delta$-thin.

Given a group $W$ and a generating set $S$, we let $\Cay(W,S)$ denote
the corresponding Cayley graph, which we regard as a geodesic metric
space by identifying each edge with a unit length interval.  Note that
the metric restricts to the {\em word metric} $d_{S}:W\times
W\rightarrow\Z_{\geq 0}$ on the vertices of the Cayley graph; that is,
for any $u,v\in W$ the distance $d_S(u,v)$ is the minimal length of a
geodesic from $u$ to $v$ in $\Cay(W,S)$.  We define the {\em length}
of an element $w\in W$ by $l(w)=d_S(1,w)$.

\begin{definition}
A finitely generated group $W$ is {\em word hyperbolic} if for some
(equivalently, any) finite generating set $S$, there exists a $\delta$ such
that $\Cay(W,S)$ is $\delta$-hyperbolic.
\end{definition}

It is known that a word hyperbolic group cannot contain a subgroup
isomorphic to $\Z\times\Z$.  For Coxeter groups, this condition is
also sufficient.

\begin{proposition}\rfthmcite{Corollary~12.6.3}{davisbook}
A Coxeter group $W$ is word hyperbolic if and only if it contains no
subgroup isomorphic to $\Z\times\Z$.
\end{proposition}

In particular, if a Coxeter group $W$ is a lattice in the isometry
group of $\fH$ (for example, a hyperbolic polygon group),
then $W$ is word hyperbolic.  

The key property of hyperbolic groups that we shall need is the
\emph{fellow-traveler property.}  For a group $W$ with generating set
$S$, we let $S^*$ denote the language of all words over the alphabet
$S$.  Any word $\alpha\in S^*$ determines a path in $\Cay(W,S)$ that
starts at the identity vertex $1\in W$.  We let $|\alpha|$ denote the
length of this path and $\bar{\alpha}\in W$ denote the terminal
vertex.  Keeping the terminology for Coxeter groups, we say that a
word $\alpha\in S^*$ is an {\em expression for $w\in W$} if
$\bar{\alpha}=w$.  (Note that $S=S^{-1}$, so that every element of $W$
is represented by some $\alpha\in S^*$, i.e.~so that the map
$\alpha\mapsto\bar{\alpha}$ from $S^*$ to $W$ is surjective.)  An
expression $\alpha$ for $w$ is a {\em reduced expression for $w$} if
the corresponding path in $\Cay(W,S)$ is a minimal length geodesic
between $1$ and $w$.  In other words, $\alpha$ satisfies
$|\alpha|=d_S(1,w)=l(w)$ and $\bar{\alpha}=w$; in particular, there is
no conflict between the use of the notation $l (w)$ to mean distance
from $1$ to $w$ in $\Cay(W,S)$ and to mean the length of a reduced
expression for $w$.

A subset $\ttL\subseteq S^*$ is called a {\em normal form} for $W$ if
the map $\alpha\mapsto\bar{\alpha}$ from $\ttL$ to $W$ is surjective.
The normal form we are interested in most is the geodesic normal form,
denoted by $\Reduced(W)$, consisting of all reduced expressions for all
elements in $W$.  More generally, for any subset $X\subseteq W$, we
define $\Reduced(X)$ to be the set of all reduced expressions for
elements of $X$.  

Two words $\alpha,\beta\in S^*$ with $|\alpha|\leq|\beta|$,
can be written uniquely as $\alpha=s_1s_2\cdots s_n$ and $\beta=t_1t_2\cdots
t_{n+p}$ where $s_i,t_j\in S$.  We say that $\alpha$ and $\beta$ are
{\em (synchronous) $k$-fellow-travelers} if 
$d_S(\bar{s_1\cdots s_i},\bar{t_1\cdots t_i})\leq k$ for all
$i=1,\ldots,n$ and $d_S(\bar{s_1\cdots s_n},\bar{t_1\cdots
t_{n+i}})\leq k$ for $i=1,\ldots,p$.  In other words, the corresponding
paths for $\alpha$ and $\beta$ in the Cayley graph are never more
than $k$-apart.  

\begin{definition}
Given a group $W$ with generating set $S$, a normal form
$\ttL\subseteq S^*$ is said to have the {\em fellow-traveler
property} (respectively, {\em two-sided fellow-traveler property}) if
there exists a $k>0$ such that for any $\alpha,\beta\in\ttL$ with
$\bar{\alpha}=\bar{\beta t}$ for some $t\in S\cup\{1\}$ (respectively,
$\bar{\alpha}=\bar{s\alpha t}$ for some $s,t\in S\cup\{1\}$), the
words $\alpha$ and $\beta$ are $k$-fellow-travelers.
\end{definition}  

\begin{remark}
The (two-sided) fellow-traveler property for a normal form $\ttL$ is
known to be equivalent to $W$ having an {\em automatic structure}
(resp., {\em biautomatic structure}) with respect to $\ttL$ in the
sense of \cite{epsteinbook}.  In particular, such a normal form must
be recognized by a finite-state automaton, hence is a regular
language.  Obviously, biautomatic implies automatic.
\end{remark}

The key fact we shall need is that word hyperbolic groups are
biautomatic {\em with respect to the geodesic normal form}.  

\begin{proposition} If $W$ is word hyperbolic, and $S$ is any finite
generating set, then $\Reduced(W)$ has the two-sided fellow-traveler
property.
\end{proposition}

\begin{proof}{} 
The two-sided fellow-traveler property is equivalent to both the
normal form and its inverse language having the (one-sided)
fellow-traveler property \cite[Definition~2.5.4 and
Lemma~2.5.5]{epsteinbook}.  Since the geodesic language is closed
under taking inverses, it is enough to show that $\Reduced(W)$
satisfies the fellow-traveler property, and this is well-known
\cite[Theorem~3.4.5]{epsteinbook}.
\end{proof} 

\section{Hyperbolic polygon cells and regular languages}\label{s:mainresults}

In this final section we prove our main results, Theorems
\ref{thm:avoidthm} and \ref{thm:biaut}.  The first uses only the
one-sided fellow traveler property, but the second requires the
stronger two-sided property.  We then combine them with Conjectures
\ref{conj:2sidedcells} and \ref{conj:1sidedcells} to deduce the
regularity of certain Kazhdan--Lusztig cells.

Given a group $W$ and finite generating set $S$, let $\mu$ be any
reduced word in $S^*$.  As above, we write $w=u.v$ in $W$ if $w=uv$
and $l(w)=l(u)+l(v)$.  We then define the subset $X_{\mu}$ of $W$ by
\[X_{\mu}=\{w\in W\;|\; w=u.\bar{\mu}.v\;\mbox{for some $u,v\in
W$}\}.\]
In other words, $X_{\mu}$ consists of all elements of $W$ that have
{\em some} reduced expression containing $\mu$ as a (consecutive)
subword.  The language $\Reduced(X_{\mu})$ therefore consists of all
reduced expressions that are equivalent to a reduced expression
containing $\mu$ as a subword. 

\begin{theorem}\label{thm:avoidthm}
Let $W$ be a word hyperbolic group, let $S$ be any finite generating set
$S$ satisfying $S=S^{-1}$, and let $\mu$ be any word in $\Reduced(W)$.
Then $\Reduced(X_{\mu})$ is a regular language. 
\end{theorem}

\begin{proof}{}
Since $W$ is word hyperbolic, $\Reduced(W)$ is a regular language.
The sublanguage $\Reduced_{\mu}(W)$ consisting of all reduced words
that contain $\mu$ as a subword (i.e., that match the regular
expression $.\!*\!\mu.*$) is also a regular language.

Now let $A$ be a finite state automaton accepting $\Reduced(W)$, and
let $k$ be a positive integer such that $\Reduced(W)$ has the
$k$-fellow-traveler property.  Let $N_k$ be the set of all reduced
expressions in $S^*$ with length $\leq k$.   Then the standard
automaton $M_{\epsilon}$ based on $(A,N_k)$ 
(see \cite[Definition~2.3.3]{epsteinbook}) accepts the language 
\[\ttL=\{(\alpha,\beta)\in\Reduced(W)^2\;|\; \mbox{
$\bar{\alpha}=\bar{\beta}$ and $\alpha$ and $\beta$ are $k$-fellow
travelers}\},\]
which is therefore regular.
But since $\Reduced(W)$ satisfies the $k$-fellow-traveler property,
this language consists precisely of pairs of reduced expressions
having the same image in $W$, i.e.,
\[\ttL=\{(\alpha,\beta)\in\Reduced_S(W)\times\Reduced_S(W)\;|\;\bar{\alpha}=\bar{\beta}\}.\]
The language $\Reduced(X_{\mu})$ is obtained by intersecting
$\ttL$ with the language
$\Reduced(W)\times\Reduced_{\mu}(W)$ and then projecting onto the first
factor.  By standard predicate calculus for regular languages (see,
e.g., \cite[Theorem~1.4.6]{epsteinbook}), the language
$\Reduced(X_{\mu})$ is therefore regular. 
\end{proof}

\begin{corollary}
Let $W$ be a hyperbolic polygon group, and let $C$ be a conjectural
two-sided cell in the decomposition $\sC$ of Conjecture
\ref{conj:2sidedcells}.  Then the language $\Reduced(C)$ is regular.
\end{corollary}

\begin{proof} 
For each finite dihedral subgroup $D$, let $\mu_D$ be any reduced
expression for the longest element $w_D$.  Then 
\[\Reduced(C_m)=\bigcup_{\bar{\mu}_D\in W_m}\Reduced(X_{\mu_D}),\]
and for $1\leq i<m$,
\[\Reduced(C_i)=\bigcup_{\bar{\mu}_D\in W_i}\Reduced(X_{\mu_D})
\setminus\bigcup_{i<j\leq m}\Reduced(C_j).\]
Since these are all obtained using finite unions, complements, and
intersections of regular languages, they are regular.  For
$\Reduced(C_{\text{id}})$ and $\Reduced(C_0)$, we note that the former
is finite, and the latter is the complement of
$\Reduced(C_{\text{id}})\cup\Reduced(C_1)\cup\cdots\cup\Reduced(C_m)$
in $\Reduced(W)$.  It follows that both are regular as well.
\end{proof}

\begin{theorem}\label{thm:biaut}
Let $W$ be a word hyperbolic group, and let $S$ be any generating
set.  Suppose $X\subseteq W$ is such that $\Reduced(X)$ is a regular
language.  Then for any $w\in W$, the language $\Reduced(w\cdot X)$ is also
regular.
\end{theorem}

\begin{proof}
In fact, the theorem holds for any normal form on a group that
satisfies the two-sided fellow-traveler property (i.e., is
biautomatic).  The proof is fairly immediate from the definitions; a
reference is \cite[Lemma~1.2]{NR}.
\end{proof}

\begin{corollary}
Let $W$ be a hyperbolic polygon group and let $w\cdot U^T$ (for
$w\in\Omega_i^{\circ}$ and $T\in\sT_i$) be one of the conjectured
one-sided cells in $C_i$.  Then $\Reduced(w\cdot U^T)$ is regular. 
\end{corollary}

\begin{proof}
First, we claim that the language $\Reduced(W^T)$ is regular.  This is
easily seen using the \emph{canonical automaton} $\cA_{\text{can}}$
that accepts $\Reduced (W)$ \cite[Theorem 4.8.3]{bb}.  The states of
this automaton, all of which are accepting, are given by the regions
that are the connected components of the complement of the hyperplane
arrangement determined by the \emph{small roots} \cite[\S4.7]{bb}.
Since the simple roots are small, the subset of the tessellation of
$\fH$ corresponding to $W^{T}$ is given by a union of states of
$\cA_{\text{can}}$.  Hence we can make an automaton accepting
$\Reduced(W^T)$ by starting with $\cA_{\text{can}}$ and only making
certain states accepting.  Thus $\Reduced(W^T)$ is regular.  Since all of the
$\Reduced(C_j)$ are regular, it follows that $\Reduced (U^T)$ is regular.  Hence,
by Theorem~\ref{thm:biaut}, $\Reduced (w\cdot U^T)$ is also regular.
\end{proof}

\begin{remark}\label{rem:otherstuff}
Conjectures \ref{conj:1sidedcells} and \ref{conj:2sidedcells} are true
for right-angled polygon groups by \cite{misha}, and more generally
for polygons with equal angles that satisfy the crystallographic condition
by \cite[\S 4]{BG1}.  Thus we can apply the results of this section to
show regularity of the languages attached to the cells for those groups.
\end{remark}

\bibliographystyle{amsplain_initials_eprint}
\bibliography{avoid}
\end{document}